\documentclass[envcountsect]{svjour3}                     
\usepackage{indentfirst}
\usepackage{graphicx}
\usepackage{mathrsfs}
\usepackage{amsmath}
\usepackage{graphicx}
\usepackage{bm}
\usepackage{caption}
\usepackage{float}
\usepackage{amssymb}

\usepackage{amsthm}
\usepackage{mathptmx}     
\usepackage{enumerate}
\usepackage{subfigure}
\usepackage{algorithm}  
\usepackage{algorithmicx}  
\usepackage{algpseudocode} 
\usepackage{setspace} 
\usepackage{subfigure}
\usepackage{authblk}
\usepackage{eucal}
\usepackage{pgf}
\usepackage{etex}
\usepackage{longtable}
\usepackage{bbding}
\usepackage{pifont}
\usepackage{multicol}
\usepackage{xspace}
\usepackage{color}
\usepackage[numbers]{natbib}
\usepackage{algorithmicx}
\usepackage{algpseudocode}
\usepackage{algorithm}
\usepackage{todonotes}

\algdef{SE}[DOWHILE]{Do}{doWhile}{\algorithmicdo}[1]{\algorithmicwhile\ #1}%

\definecolor{dkgreen}{rgb}{0,0.6,0}
\definecolor{gray}{rgb}{0.5,0.5,0.5}
\definecolor{mauve}{rgb}{0.58,0,0.82}
\usepackage{hyperref}
\hypersetup{
	colorlinks=true,
	linkcolor=blue,
	filecolor=magenta,
	urlcolor=cyan,
	citecolor=blue
}
\usepackage{listings}
\usepackage{lineno,hyperref}
\usepackage[capitalize]{cleveref}

\newcommand{\MIP}{\text{MIP}\xspace}

\newcommand{\rev}[1]{{\color{black}#1}}

\newcommand{\I}{\mathcal{I}}

\newcommand{\C}{\mathcal{C}}
\newcommand{\CO}{\mathcal{O}}
\newcommand{\M}{\mathcal{M}}
\newcommand{\N}{\mathcal{N}}
\newcommand{\R}{\mathcal{R}}
\newcommand{\X}{\mathcal{X}}
\newcommand{\Impl}{\texttt{\#Bchg}\xspace}
\newcommand{\Tprobing}{\texttt{T$_{\text{probing}}$}}
\newcommand{\HiGHS}{{{HiGHS}}\xspace}

\newcommand{\nnz}{\texttt{nnz}\xspace}

\crefname{theorem}{Theorem}{Theorems}
\crefname{example}{Example}{Examples}
\crefname{observation}{Observation}{Observations}
\crefname{remark}{Remark}{Remarks}
\crefname{proposition}{Proposition}{Propositions}
\crefname{property}{Property}{Propertys}
\crefname{lemma}{Lemma}{Lemmas}
\crefname{corollary}{Corollary}{Corollaries}
\crefname{algocf}{Algorithm}{Algorithms}	
\crefname{table}{Table}{Tables}	
\crefname{figure}{Figure}{Figures}
\crefname{algorithm}{Algorithm}{Algorithms}
\crefname{section}{Section}{Sections}

\newcommand{\Default}{\texttt{Default}\xspace}
\newcommand{\DFProbing}{\texttt{DF+Probing}\xspace}
\newcommand{\All}{\texttt{All}\xspace}
\newcommand{\GDF}{\texttt{GDF}\xspace}

\hypersetup{
	colorlinks=true,
	linkcolor=blue,
	filecolor=magenta,
	urlcolor=cyan,
	citecolor=blue
}

\usepackage{titlesec}
\titlespacing*{\section}
{0pt}{3.0ex plus 1ex minus .1ex}{1.5ex plus .2ex}
\titlespacing*{\subsection}
{0pt}{2.0ex plus 1ex minus .1ex}{1.5ex plus .2ex}

\usepackage{orcidlink}
\definecolor{orcidlogocol}{HTML}{A6CE39}
\tikzset{
	orcidlogo/.pic={
		\fill[orcidlogocol] svg{M256,128c0,70.7-57.3,128-128,128C57.3,256,0,198.7,0,128C0,57.3,57.3,0,128,0C198.7,0,256,57.3,256,128z};
		\fill[white] svg{M86.3,186.2H70.9V79.1h15.4v48.4V186.2z}
		svg{M108.9,79.1h41.6c39.6,0,57,28.3,57,53.6c0,27.5-21.5,53.6-56.8,53.6h-41.8V79.1z M124.3,172.4h24.5c34.9,0,42.9-26.5,42.9-39.7c0-21.5-13.7-39.7-43.7-39.7h-23.7V172.4z}
		svg{M88.7,56.8c0,5.5-4.5,10.1-10.1,10.1c-5.6,0-10.1-4.6-10.1-10.1c0-5.6,4.5-10.1,10.1-10.1C84.2,46.7,88.7,51.3,88.7,56.8z};
	}
}

\newcommand\orcidicon[1]{\href{https://orcid.org/#1}{\mbox{\scalerel*{
				\begin{tikzpicture}[yscale=-1,transform shape]
					\pic{orcidlogo};
				\end{tikzpicture}
			}{|}}}}

\def\useLongFormat{0} 

\begin{document}
\setlength{\abovedisplayskip}{0.11cm}
\setlength{\belowdisplayskip}{0.11cm}

\title{Enhancing Presolve in Mixed Integer Programming by Combining Probing and Dual Fixing}

\author{Zhao-Wei Wang \and 
        Wei-Kun Chen \and
        Yu-Hong Dai}
\institute{Zhao-Wei Wang\orcidlink{0000-0002-7165-9895} and Yu-Hong Dai\orcidlink{0000-0002-6932-9512} \at Academy of Mathematics and Systems Science, Chinese Academy of Sciences, 100190, Beijing, China,
           School of Mathematical Sciences, University of Chinese Academy of Sciences, 100049, Beijing, China \\
           \email{\{wangzhaowei, dyh\}@lsec.cc.ac.cn} 
        \and
           Wei-Kun Chen\orcidlink{0000-0003-4147-1346} \at School of Mathematics and Statistics, Beijing Institute of Technology, 100081, Beijing, China \\
           \email{chenweikun@bit.edu.cn}
           }

\date{Received: date / Accepted: date}
\maketitle

\begin{abstract}
    Probing and dual fixing are two powerful presolve techniques in mixed integer programming (MIP) solvers. 
    Probing tentatively sets some  binary variables to $0$ or $1$,
    {applies} linear constraint based domain propagation techniques to derive better variable bounds, and {extracts} useful information such as stronger variable implications and better global variable bounds.
	Dual fixing attempts to fix variables to lower or upper bounds while ensuring {that} at least one optimal solution is retained, as long as the problem was feasible.
	In this paper, we investigate how to combine the two approaches in \MIP solvers to achieve a better performance.
    In particular, we first embed dual fixing into the probing framework, deriving more useful variables' implications for enhancing the capability of probing. 
    Then, we develop  an improved dual fixing technique where more variable fixings can be applied, and use  the probing framework to detect the reductions.
    Computational results on the MIPLIB 2017 benchmark instances demonstrate \rev{the potential of the two proposed techniques in combining probing and dual fixing on the open-source \MIP solver HiGHS.}
\end{abstract}


\section{Introduction}\label{section-introduction}

	Presolve is a fundamental technique for mixed integer programming (\MIP) problems.
	It  contains a set of algorithms  designed to eliminate redundant information and strengthen the formulation, with the goal of accelerating the subsequent branch-and-cut process.
	In the literature, many researchers have proposed various presolve techniques to improve the computational performance of solving \MIP{s}; see \cite{Andersen1995,Brearley1975a,Crowder1983,Gamrath2015,Gemander2020a,Gleixner2023,Guignard1981,Hoffman1991,Johnson1980,Savelsbergh1994} among many of them. 
	Achterberg et al. \citep{Achterberg2019} provided a comprehensive discussion of various presolve techniques.
	According to the investigations in  \cite{Achterberg2019,Achterberg2013b}, presolve can frequently turn a problem from intractable to easily solvable, and therefore, it 
	 has become one of the most important ingredients of modern MIP 
	solvers.


    \emph{Probing} is a well-known presolve technique for \MIP problems \citep{Savelsbergh1994}. 
    Its basic idea is to tentatively set some binary variables to $0$ or $1$, {and}
    apply linear constraint based domain propagation techniques to derive better variable bounds.
    Analyzing these results can provide useful information including variable fixings, substitutions,  implications (which are useful for the subsequent branch-and-cut process), and better global variable bounds.
    Probing is a very effective presolve technique for improving the performance of \MIP solvers and has been widely investigated in the literature; see, e.g.,  \cite{Achterberg2007,Achterberg2019,Dai2024,Gleixner2023}.
    Kılınç et al. \citep{Klnc2018} and Ma et al. \citep{Ma2024} extended probing to solving mixed integer nonlinear programming and stochastic programming problems.
    

    \emph{Dual fixing} is another well-known presolve technique in the literature.
    It is a \emph{\rev{dual} presolve technique}, which derives reductions by considering the information of the objective function; the reductions may remove optimal solutions, while ensuring that at least one optimal solution is retained, as long as the original problem was feasible \cite{Achterberg2007,Achterberg2013b,Heinz2013}.
    In particular, dual fixing sets a variable to its upper bound (respectively, lower bound) if (i)  increasing (respectively, decreasing) the variable's value makes all the constraints other than its bound constraints in the problem less tight, and
    (ii) the variable has a nonpositive (respectively, nonnegative) objective coefficient.
    \rev{If the variable {being fixed} has a nonzero objective coefficient, then the resulting reduction is a \emph{weak} dual presolve reduction, which may remove only suboptimal solutions while retaining all optimal solutions.
    Otherwise, the reduction is a \emph{strong} dual presolve  reduction, which may remove optimal solutions while  retaining at least one optimal solution \cite{Hoen2024}.}
    Dual fixing is simple to implement and has been included as a standard presolve ingredient  in all state-of-the-art \MIP solvers \cite{Achterberg2007,Achterberg2019,Achterberg2013b}.

    Although probing and dual fixing are well studied, the combination of the two approaches, however, has not been discussed in the literature. 
    In this paper, we attempt to fill this research gap by developing methods to combine and enhance these two powerful techniques for \MIP problems.
    In  particular, we first embed dual fixing into the probing framework to detect more reductions.
    The rationale behind this is that tentatively setting some binary variables to $0$ or $1$ and then applying domain propagation techniques could substantially change the structure of the problem. 
    This can further trigger dual fixing to detect more variable fixings, which in turn  enables probing to detect more reductions.
    
    In addition, we develop a probing-based approach to detect more dual fixings.
    Specifically, we first propose an improved dual fixing technique, which, 
    instead of relying directly on the constraint matrix of the MIP problem, relies on the point representation of the feasible region (i.e., the collection of points at which the constraints of the \MIP problem hold), and can detect more variable fixings.
    Then, we use the probing framework to quickly detect the condition for which the improved dual fixing can be applied.
    
    We implement the two proposed techniques (in combining probing and dual fixing) in the state-of-the-art {open-source} \MIP solver \HiGHS \citep{Huangfu2018}.
    Computational results on the MIPLIB 2017 benchmark instances \citep{Gleixner2021} demonstrate \rev{the potential of the two proposed techniques on HiGHS.}


    
%
%

    This article is organized as follows. 
    \cref{section-dfprobing} introduces the integration of dual fixing in the probing framework. 
    \cref{section-generalize} proposes improved conditions for applying dual fixing, and uses the probing framework to verify whether the conditions hold.
    \cref{section-num} presents the computational results.
    {\cref{section-conclusion} concludes the paper.} \\[5pt]
    \noindent {\bf {Notation}}. We use the following {notation} throughout the paper. 
    Given a matrix $A\in\mathbb{R}^{m\times n}$, vectors $c\in\mathbb{R}^n$, 
    $\ell\in(\mathbb{R}\cup\{-\infty\})^n$, $u\in(\mathbb{R}\cup\{+\infty\})^n$, $b \in \mathbb{R}^m$, 
    variables $x\in$ $\mathbb{R}^n$ with $x_j\in\mathbb{Z}$ for $j\in \I\subseteq \N := \left\{1,\cdots,n\right\}$, 
    the \MIP problem can be written as
    \begin{equation}
    	\min \left\{c^\top x ~\mid~ Ax \le b,~ \ell \le x \le u,~ {x_j \in \mathbb{Z},~ \forall~j \in \I}\right\}.
    	\label{MIP}
    	\tag{MIP}
    \end{equation}
    We denote the feasible region of \eqref{MIP} as 
    \begin{equation}\label{Xdef}
    	\X= \{ x\in \mathbb{R}^n ~ \mid ~ Ax \le b,~ \ell \le x \le u,~ {x_j \in \mathbb{Z},~ \forall~j \in \I} \}.
    \end{equation}
   For matrix $A$, the $j$-th column is denoted as $A_{\cdot j}$,
   the  $i$-th row is denoted as $A_{i\cdot}$, and
   the $(i,j)$-th entry is denoted {as} $a_{ij}$.
   {Finally, we use $\M = \left\{1,\cdots,m\right\}$ to denote the index set of constraints.}

\section{Integration of dual fixing in probing}
\label{section-dfprobing}

In this section, we first review the probing and dual fixing techniques and then enhance probing by embedding dual fixing to derive more reductions.


\subsection{Probing}\label{probing-section}
	
	We first briefly discuss the probing technique; see \cite{Achterberg2007,Savelsbergh1994} for a detailed discussion.
	The basic idea of probing is to tentatively set some  binary variables to $0$ or $1$,
	apply (linear constraint based) domain propagation techniques to derive better variable bounds, and extract useful information such as stronger variable implications and better global variable bounds.
    Let $x_k$, $k \in \N$, be a binary variable.
    For $j \in \N\backslash \{k\}$, let $\ell_j^0$ and $u_j^0$  be the lower and upper bounds of $x_j$ deduced from $x_k = 0$ (after applying the domain propagation), and $\ell_j^1$ and $u_j^1$  be the lower and upper bounds of $x_j$ deduced from $x_k = 1$.
    Then, we can derive the following information from probing {on} $x_k$:
    \begin{enumerate}
        \item[R1.] If setting $x_k = 0$ (respectively, $x_k =1$) leads to an infeasible problem, then we can fix $x_k = 1$ (respectively, $x_k =0$). 
        \item[R2.] If $\ell_j^0=u_j^0$ and $\ell_j^1 = u_j^1$, then $x_j$ can be substituted by $x_k$: $x_j:= \ell_j^0 + (\ell_j^1-\ell_j^0)x_k$.
        \item[R3.] If no variable fixing or substitution is detected, we can derive variable implications between $x_j$ and $x_k$:
            \begin{equation}
            	\label{implication}
                \begin{array}{l}
                    x_j \ge \left(\ell_j^1 - \ell_j^0\right) x_k + \ell_j^0, ~
                    x_j \le \left(u_j^1 - u_j^0\right) x_k + u_j^0. \\
                \end{array}
            \end{equation}
            These implications can be used to guide variable branching \citep{Achterberg2007}
            and 
            enhance the mixed integer rounding cut separation \citep{Marchand2001} in the subsequent branch-and-cut process.
        \item[R4.] Finally, we can set ${\ell}_j:= \min\{\ell_j^0, \ell_j^1\}$ and ${u}_j:=\max\{u_j^0, u_j^1\}$ as new global lower and upper bounds for variable $x_j$.
    \end{enumerate}
	Probing could be time-consuming, but effective termination criteria have been extensively investigated in the literature \citep{Achterberg2019}.
	Indeed, probing is a very effective presolve technique \citep{Achterberg2019, Achterberg2013b} for improving the performance of \MIP solvers and has been implemented in all state-of-the-art \MIP solvers.
    
\subsection{Dual Fixing}\label{section-dualfixing}

Dual fixing is a \rev{dual} presolve technique which attempts to fix variables to lower or upper bounds while ensuring {that} at least one optimal solution is retained, as long as the problem was feasible. 
Specifically, dual fixing considers the following two cases where variable fixing can be applied to variable $x_j$, $j \in \N$:
\begin{enumerate}[left=0pt]
    \item [(DF1)] $c_j \ge 0$, $a_{ij} \ge 0$, $\forall~ i \in \M$, and
    \item [(DF2)] $c_j \le 0$, $a_{ij} \le 0$, $\forall~ i \in \M$.
\end{enumerate}
In case (DF1), if $\ell_j > -\infty$, we can fix $x_j = \ell_j$;
if $\ell_j = -\infty$ and $c_j > 0$, then the problem is unbounded or infeasible; and 
if $\ell_j = -\infty$ and $c_j = 0$, we can remove all constraints $i \in \M$ 
with $a_{ij} \neq 0$ from the problem. 
The last case is valid since in a post-solve process, we can always recover a finite value of $x_j$ to satisfy all those constraints (as long as the problem was feasible). 
Similarly, for case (DF2), if $u_j< +\infty$, we can fix $x_j = u_j$; if $u_j = +\infty$ and $c_j < 0$, then the problem is unbounded or infeasible; and 
if $u_j = +\infty$ and $c_j = 0$, we can remove all constraints $i \in \M$ 
with $a_{ij} \neq 0$ from the problem. 
\rev{Note that if $c_j \neq 0$, then the reductions derived from case (DF1) or (DF2) may remove only suboptimal solutions while retaining all optimal solutions, and are therefore weak dual reductions;
otherwise, these reductions may remove optimal solutions while ensuring that at least one optimal solution is retained, and are therefore strong dual reductions \cite{Hoen2024}.}

Dual fixing can be easily implemented with the {concepts} of   \texttt{down-locks} and  \texttt{up-locks} \citep{Achterberg2007}. 
Formally, for each variable $x_j,~ j \in \N$, letting
\begin{equation*}
	\C_j^- = \left\{i \in \M ~\mid~ a_{ij} < 0 \right\}~\text{and}~\C_j^+ = \left\{i \in \M ~\mid~ a_{ij} > 0 \right\},
\end{equation*}
then the  \texttt{down-locks} and \texttt{up-locks} of variable $x_j$ are defined by $|\C_j^-|$ and $|\C_j^+|$. 
Intuitively, $|\C_j^-|$ and $|\C_j^+|$ are the {numbers} of constraints that ``block'' variable $x_j$ from decreasing and increasing, respectively.
With these definitions, if $|\C_j^-| = 0$ and $c_j \ge 0$ or $|\C_j^+| = 0$ and $c_j \le 0$, then dual fixing can be applied.
This also implies that dual fixing can be implemented with the complexity of $\CO(\nnz)$, where \nnz is the number of nonzeros in the constraint matrix $A$.


Achterberg et al. \citep{Achterberg2019} further considered the case $c_j \geq 0$ and $\C_j^-=\{r\}$ (i.e., only the $r$-th constraint blocks variable $x_j$ from decreasing), and proposed the dual substitution technique to derive variable substitutions.
In particular, for a binary variable $x_k$, if $x_k = 0$ implies the redundancy of constraint $r$ and thus {$x_j  = \ell_j$} (by dual fixing), and $x_k=1$ implies $x_j =u_j$ (e.g., by domain propagation techniques), then we can substitute $x_j:= \ell_j + (u_j - \ell_j)x_k$. 
A similar argument can be applied for the case $c_j \leq 0$ and $|\C_j^+|=1$.
In the next subsection, we will show that reductions detected by dual substitution, however, can be automatically detected if dual fixing is embedded into the probing framework.

%

\subsection{Embedding dual fixing into the probing framework}\label{section-embed}


The motivation of embedding dual fixing into the probing framework is that setting some binary variables to $0$ or $1$ and then applying (linear constraint based) domain propagation techniques could substantially change the structure of the problem; that is, some variable bounds become much tighter and therefore, some constraints may become redundant.
This can further trigger dual fixing to detect more variable fixings, thereby enabling probing to extract more useful information; see R1--R4 in \cref{probing-section}.
This is formalized below.

For a binary variable $x_k$, we probe {on} $x_k = v \in \left\{0, 1\right\}$, obtaining the subproblem:
\begin{equation}
    \min \left\{c^\top x ~|~ Ax \le b,~ \ell \le x \le u,~ {x_j \in \mathbb{Z},~ \forall~j \in \I},~ {x_k = v} \right\}.
    \label{Reduced-MIP}
    \tag{\MIP-$v$}
\end{equation}
Applying domain propagation techniques on problem \eqref{Reduced-MIP}, the variable bounds $\ell$ and $u$ can be tightened and some constraints will therefore become redundant. 
In particular, let
\begin{equation}
    \min \left\{c^\top x ~|~ A_{\R \cdot} x \le b_\R,~ \ell' \le x \le u',~ {x_j \in \mathbb{Z},~ \forall~j \in \I},~{x_k=v} \right\},
    \label{RemoveRedundant-MIP}
\end{equation}
be the (equivalent) reduced \MIP problem, where $\ell'$ and $u'$ are the variable bounds in the reduced problem, $\R$ is the set of non-redundant  constraints, 
$A_{\R \cdot}$ denotes the matrix consisting of the coefficients of $A$ restricted to the row indices in $\R$, and $b_\R$ denotes the vector  $b$ restricted to the indices in $\R$.
For $j\in \N$, we can update the sets corresponding to \texttt{down-locks} and \texttt{up-locks} as follows:
\begin{equation*}
    \overline{\C}_j^- = \left\{i \in \R ~\mid~ a_{ij} < 0 \right\}~\text{and}~\overline{\C}_j^+ = \left\{i \in \R ~\mid~ a_{ij} > 0 \right\}.
\end{equation*}
Then, {in the reduced \MIP problem \eqref{RemoveRedundant-MIP}}, we can set
$x_j = {\ell_j^\prime}$ if $\left|\overline{\C}_j^-\right| = 0$ and $c_j \ge 0$, or set $x_j = {u_j^\prime}$
if $\left|\overline{\C}_j^+\right| = 0$ and $c_j \le 0$. 
Note that $\left|\overline{\C}_j^-\right| \le \left|\C_j^-\right|$ and $\left|\overline{\C}_j^+\right| \le \left|\C_j^+\right|$, and hence applying dual fixing on {the reduced \MIP problem \eqref{RemoveRedundant-MIP}} may detect more variable fixings, obtaining tighter lower bounds $\ell^v$ and upper bounds $u^v$ ({valid for \eqref{RemoveRedundant-MIP}}). 
As a result, after probing {on} $x_k =0$ and $x_k =1$, we can extract more useful information from R1--R4 in \cref{probing-section}, {which is globally valid for problem \eqref{MIP}}. 
The overall procedure of embedding dual fixing into the probing framework is summarized in \cref{alg1}.
\begin{algorithm}[t]
    \caption{Dual fixing augmented probing}
    \small
    \begin{algorithmic}[1]
        \Require A binary variable $x_k$ to be probed {on}, $A,~b,~c,~\ell,~u,~\I$.
        \For{$v = 0, 1$}
        \State Set $x_k = v$ in problem \eqref{MIP}, obtaining subproblem \eqref{Reduced-MIP}.
        \Repeat
        	\State Call domain propagation presolver on problem \eqref{Reduced-MIP} to obtain variable bounds $\ell' \leq x \leq u'$.
        	\State Call redundant row presolver to compute the set $\R$ of non-redundant constraints.
        	\State Call dual fixing presolver on problem \eqref{RemoveRedundant-MIP} to obtain tighter variable bounds $\ell^v \leq x \leq  u^v$.
        \Until{no reduction is found.}
        \EndFor
        \State Extract reductions using R1--R4.
    \end{algorithmic}
    \label{alg1}
\end{algorithm}
In \cref{alg1}, the loop in steps 3--7 iteratively applies linear constraint based domain propagation and dual fixing techniques until no more reduction is found (based on our computational experience, it usually terminates within a few iterations). In addition, in the loop,
step 5 computes the set $\R$ of non-redundant constraints and step 6 applies the dual fixing presolver on the reduced {\MIP} problem \eqref{RemoveRedundant-MIP}. 
Both steps can be implemented with the complexity of $\CO(\nnz)$.

We use the following example to illustrate the effectiveness of the proposed dual fixing augmented probing framework.


\begin{example}\label{eexMIP}
Consider the following example
\ifnum\useLongFormat=1
\begin{equation}\label{exMIP}
	\begin{array}{ll}
		\min & x_1 + x_2 \\
		~\text{s.t.} & - x_1 + x_3 \le 0 \\
		& - x_2 + x_3 + x_4 \le 1 \\
		& x_1, x_2, x_3, x_4 \in \left\{0, 1\right\}.
	\end{array}
\end{equation}
\else
\begin{equation}\label{exMIP}
	\min \{   x_1 + x_2 ~\mid~ - x_1 + x_3 \le 0,~ - x_2 + x_3 + x_4 \le 1,~x_1, x_2, x_3, x_4 \in \left\{0, 1\right\}\}.
\end{equation}
\fi
By simple computation, it follows that $\C_1^- = \left\{1\right\}$ and $\C_2^- = \left\{2\right\}$, and thus the \texttt{down-locks} of both variables $x_1$ and $x_2$ are equal to $1$, prohibiting us from fixing the two variables to $0$.
If we probe {on} $x_3 = 0$, the first constraint $- x_1 + x_3 \le 0$ becomes redundant. We can update $\overline{\C}_1^- = \varnothing$ and apply dual fixing to fix $x_1 = 0$.
If we probe {on} $x_3 = 1$, then by the domain propagation, constraints $-x_1 + x_3 \leq 0$ and $x_1 \in \{0,1\}$ will force $x_1 = 1$. 
As a result, we can substitute $x_1 = x_3$ in problem \eqref{exMIP}. 
This is the dual substitution technique provided in \cite{Achterberg2019}; see also the last paragraph of \cref{section-dualfixing}.
Moreover, if we probe {on} $x_4 = 0$, the second constraint $-x_2 + x_3 + x_4 \le 1$ is redundant, and thus we can apply dual fixing to fix $x_2 = 0$. 
However, probing {on} $x_4 = 1$ cannot force $x_2 =1$ (by domain propagation techniques), and hence the dual substitution cannot substitute $x_2$ out
from the problem. 
Nevertheless, using the proposed dual fixing augmented probing framework in \cref{alg1}, we can detect the variable implication $x_2 \le x_4$, 
which can be used in R3 of \cref{probing-section} to detect further reductions.
\end{example}

{Note that when probing on multiple binary variables,  special care must be taken on the implications \eqref{implication} extracted from the tightened variable bounds (i.e., $\ell^0$, $u^0$, $\ell^1$, and $u^1$) and for which the reasoning comes from dual fixing of variable $x_j$ with $c_j=0$ in step 6 of \cref{alg1}. 
Indeed, if $c_j >0$ or $c_j <0$, then the implications of variable $x_j$ in \eqref{implication} obtained by dual fixing will not remove any optimal solution of problem \eqref{MIP}.
However, if $c_j =0$, then the implications of variable $x_j$ in \eqref{implication} obtained by dual fixing may exclude some optimal solution(s) of  problem \eqref{MIP}.
{Such implications, obtained by probing on variable $x_k$, could be inconsistent with the subsequent reductions derived by probing on other variables, as simultaneously adding them into problem \eqref{MIP} may exclude all optimal solutions. 
The following example illustrates this inconsistency.}}
\begin{example}\label{inconsistency}
	{Consider the following example
	\begin{equation}\label{MIPex}
		\min\{ -x_1 - x_2  ~\mid~  x_1 + x_2 + x_3 \ge 1,~  x_3 - x_2 \le 0,~x_1, x_2, x_3 \in \left\{0,1\right\} \}.
	\end{equation}
	By simple calculation, the optimal solutions of problem \eqref{MIPex} are $\left(1, 1, 0\right)$ and $\left(1, 1, 1\right)$.
	Probing on $x_1 = 1$, {constraint $x_1 + x_2 + x_3 \geq 1$ becomes redundant and we can fix $x_3 = 0$ by dual fixing;
  	probing on $x_1 = 0$ does not yield additional reductions.
    Using R3 in \cref{probing-section}, we obtain the implication $ x_3 \leq 1-x_1$.
	On the other hand, probing on $x_2 = 1$, constraint  $x_3 - x_2 \leq 0$ becomes redundant, and we can fix $x_3 = 1$ by dual fixing;
    probing on $x_2 = 0$, we can fix $x_3 = 0$ by domain propagation.
    Using R2 in \cref{probing-section}, we obtain the variable substitution $x_3 = x_2$.}
	Note that the implication $ x_3 \leq 1-x_1$ (obtained by probing on $x_1$) excludes the optimal solution $(1,1,1)$ but is valid for problem \eqref{MIPex}, 
 	as the optimal solution $(1,1,0)$ is still retained after adding it.
	However, the implication $ x_3 \leq 1-x_1$ is inconsistent with the variable substitution $x_3 = x_2$ (obtained by probing on $x_2$) as simultaneously adding it and  $x_3 = x_2$ into problem \eqref{MIPex} will eliminate the two optimal solutions $\left(1, 1, 0\right)$ and $\left(1, 1, 1\right)$.}
\end{example}%


The above example illustrates that implications \eqref{implication}, for which the reasoning comes from dual fixing of variable $x_j$ with $c_j=0$,  through probing on a variable cannot be directly used, since they may, together with the subsequent reductions derived by probing on other variables, exclude all optimal solutions of problem \eqref{MIP}.
To avoid this interference caused by probing on different variables,
one approach is to ignore variables $x_j$ with $c_j = 0$ when applying dual fixing in step 6 of \cref{alg1}.
With this approach, all optimal solutions will be retained, and the resultant method becomes a weak dual presolve technique; see the third paragraph in \cref{section-introduction}.
However, based on our computational experience, this approach is too conservative and may underestimate the potential of  the dual fixing augmented probing.

To enable the application of more reductions, we propose to disregard only the implications (2) extracted from the tightened variable bounds (i.e., $\ell^0$, $u^0$, $\ell^1$, and $u^1$) and for which the reasoning comes from dual fixing of variable $x_j$ with $c_j=0$ in step 6 of \cref{alg1}. 
Specifically, in \cref{alg1}, we (i) temporarily store the bound changes of $\ell^v$ and $\ell$ (respectively, $u^v$ and $u$ where $v \in \{0,1\}$) for which the reasoning comes from dual fixing of variable $x_j$ with $c_j=0$, (ii) use them, together with other bound changes, to perform the global reductions R1, R2, and R4 in step 9, 
and (iii) disregard them right after probing on variable $x_k$.

%

{Note that in the implementation of the probing in \MIP solvers (e.g.,  \HiGHS), the variable fixing and bound tightening in R1 and R4 are performed right after probing on the current variable $x_k$,
but the variable substitutions in R2 are delayed until all candidate variables in the current round are probed on, as to avoid the computational overhead spent in the initialization of probing \citep{Gleixner2023}.
Therefore, inconsistency may still exist, as illustrated by the following example.
\begin{example}\label{inconsistency-delay}
	{Consider a variant of problem \eqref{MIPex}: 
	\begin{equation}\label{MIPex-delay}
		\min\{ -x_1 - x_2  ~\mid ~ x_1 + x_2 + x_3 \ge 1,~  x_3 - x_2 \le 0,~ {x_3 \geq 1 - x_1},~ x_1, x_2, x_3 \in \left\{0,1\right\} \},
	\end{equation}
    where an additional constraint $x_3 \ge 1 - x_1$ is enforced.
    The optimal solutions of \eqref{MIPex-delay} are still $\left(1, 1, 0\right)$ and $\left(1, 1, 1\right)$.
    Probing on $x_1$ and using R2 in \cref{probing-section},  we obtain the variable substitution $x_3 = 1 - x_1$.
    Probing on $x_2$ and using R2,  we obtain the variable substitution $x_3 = x_2$.
    However, the two variable substitutions $x_3 = 1 - x_1$ and $x_3 = x_2$ are inconsistent, as simultaneously adding both into problem \eqref{MIPex-delay} will eliminate the two optimal solutions $\left(1, 1, 0\right)$ and $\left(1, 1, 1\right)$.
    Note that if the variable substitution $x_3=1-x_1$ is applied right after probing on $x_1$, we obtain the reduced problem $\min\{ x_3 - x_2 - 1  ~\mid~  x_3 - x_2 \le 0,~ x_2, x_3 \in \left\{0,1\right\} \}$ for which probing on variable $x_2$ will not yield the variable substitution $x_3 = x_2$. 
    }
\end{example}%
To further avoid inconsistency incurred by the delayed variable substitution for which the reasoning comes from dual fixing of variable $x_j$ with $c_j=0$,
 we can assign a fixed direction for {variable $x_j$ with $c_j =0$ when} performing dual fixing; that is, either (DF1) or (DF2) (but not both) in \cref{section-dualfixing} can be performed in the whole probing process.
Note that this is equivalent to a perturbation on the objective coefficient of variable $x_j$ with $c_j = 0$ and therefore, the delayed variable substitutions are consistent.
In our implementation, the fixed direction for variable $x_j$ with $c_j=0$ is set to be the direction where dual fixing for variable $x_j$ was first performed.}

\begin{remark}
    {Note that after assigning fixed directions for variables with zero objective coefficients, the implications derived by dual fixing are now consistent (at the time probing was performed).
	However, these implications may still be inconsistent with other dual reduction components in the \MIP solver (such as global dual fixing).
	For this reason, we do not use this kind of implication out of the probing. }
\end{remark}

\section{Using probing to detect more (generalized) dual fixings}\label{section-generalize}

The conditions (DF1) and (DF2) in \cref{section-dualfixing} for applying dual fixing are straightforward and can be efficiently verified, but they also highly depend on the use of $A$ to characterize the feasible region $\X$ of \eqref{MIP}.
Indeed, to characterize the feasible region $\X$ of \eqref{MIP}, different formulations can be used, each of which associates with a matrix $A$; see \cite[Section 1.6]{wolsey2020}.
It is possible that in some formulations, (DF1) or (DF2) is satisfied, allowing dual fixing to be applied, while in others, neither (DF1) nor (DF2) is satisfied.
In order to remedy this, in this section, we first introduce  improved conditions for applying dual fixing, which directly take the point representation of the feasible region $\X$ into account rather than a specific constraint matrix representation, and thus enable the detection of more reductions than the classic conditions (DF1) and (DF2).
Then, we transform the problem of checking whether the conditions are satisfied into optimization problems.
Finally, we use the probing framework to solve the optimization problems, thereby enabling a quick detection of the reductions.


\subsection{Improved conditions for applying dual fixing}

To develop the improved conditions for applying dual fixing, we need the {concepts} of \emph{lower bound reachable} and \emph{upper bound reachable}, which rely on the point representation of the feasible region $\X$ of \eqref{MIP}. 
\begin{definition}\label{reachableDefinition}
	(i) A variable $x_j$ is lower bound reachable with respect to $\X$ if setting $x_j=\ell_j$ in any feasible solution $x$ of $\X$ yields another feasible solution, {that is, if} {$x=(x_1, \ldots, x_j, \ldots, x_n)$ $\in \X$}, {then} {$\bar{x}=(x_1, \ldots, \ell_j, \ldots, x_n) \in \X$ as well} (when $\ell_j= -\infty$, this means that $\bar{x}=(x_1,$ $\ldots, a, \ldots, $ $x_n) \in \X$ holds for any $a \leq x_j$).
	(ii) A variable $x_j$ is upper bound reachable with respect to $\X$ if setting $x_j=u_j$ in any feasible solution $x$ of $\X$ yields another feasible solution, i.e., if $x=(x_1, \ldots, x_j, \ldots, x_n) \in \X$, then $\bar{x}=(x_1, \ldots, u_j, \ldots, x_n) \in \X$ as well (when $u_j= +\infty$, this means that $\bar{x}=(x_1, \ldots, a, \ldots, x_n) \in \X$ holds for any $a \geq x_j$).
\end{definition}
Note that for $\X$ defined in \eqref{Xdef}, if $a_{ij}\geq 0$ holds for all $i \in \M$, then $x_j$ must be lower bound reachable; 
if $a_{ij}\leq 0$ holds for all $i \in \M$, then $x_j$ must be upper bound reachable.
Moreover, it is simple to see that if 
\begin{enumerate}[left=0pt]
	\item [(IDF1)] $c_j \ge 0$ and $x_j$ is lower bound reachable, or
	\item [(IDF2)] $c_j \le 0$ and $x_j$ is upper bound reachable,
\end{enumerate}
then dual fixing can be applied to fix variables to lower or upper bounds. 
The following example shows that the conditions in (IDF1)--(IDF2) enable the detection of more reductions than those in (DF1)--(DF2).

%

\begin{example}\label{ConditionExample}
    Consider the following example
    {
    \begin{equation}
        \label{condition-mip}
        \min\{-x_1 -2x_2 +x_3~ \mid~ 0 \le 8x_1 + 4x_2 - x_3 \le 8,~x_1, x_2, x_3 \in \left\{0, 1\right\} \}.
    \end{equation}}%
	As {$|\C^-_3| = 1$}, dual fixing cannot be applied to fix {$x_3 = 0$}. 
	However, as
	{$\X = \{ (x_1, x_2, x_3) \in \{0,1 \}^3 ~ \mid ~ 0 \le 8x_1 + 4x_2 - x_3 \le 8 \}= \left\{(1, 0, 1), (1, 0, 0), (0, 1, 1), 
	(0, 1, 0),(0,0,0)\right\}$},
	$x_3$ is {lower} bound reachable, and by {$c_3 = 1 \geq 0$}, {condition (IDF1)} is satisfied, enabling the application of dual fixing to fix $x_3 = 0$. 
\end{example}

\subsection{Determining whether a variable is lower or upper bound reachable}

{Using \cref{reachableDefinition}, we can provide a natural approach to determine whether a variable $x_j$ is lower or upper bound reachable with respect to $\X$---one can enumerate all feasible points of $\X$ and check whether the condition for a variable to be lower or upper bound reachable is satisfied.
This approach, however, is computationally intractable, as the number of feasible points in $\X$ could be infinite}. 
In order to develop a {computationally} tractable approach, we need the following result.

\begin{theorem}
	\label{fixing-theorem}
	{Let $o_j^{-,r}$ and $o_j^{+,r}$ be defined as
	\begin{align}
		& o_j^{-,r} := \max_{x_j > \ell_j} \left\{\sum\limits_{s \in \N \backslash \{j\}} a_{rs} x_s ~\mid~ x \in \X \right\},~\forall~j \in \N,~r \in \M ~\text{with}~a_{rj} <0,\label{opt1}\\
		& o_j^{+,r} := \max_{x_j < u_j}    \left\{\sum\limits_{s \in \N \backslash \{j\}} a_{rs} x_s ~\mid~ x \in \X \right\},~\forall~j \in \N,~r \in \M ~\text{with}~a_{rj} >0.\label{opt2}
	\end{align}}
    Then for $j \in \N$, (i) $x_j$ is lower bound reachable if and only if $o_j^{-,r} + a_{rj} \ell_j  \le b_r$  holds for all $r \in \M$ with $a_{rj} < 0$; 
    and (ii)  $x_j$ is upper bound reachable if and only if $	o_j^{+,r}+ a_{rj} u_j  \le b_r$ holds for all $r\in \M$ with $a_{rj} > 0$.
\end{theorem}

\begin{proof}
	We only prove statement (i) as the proof of statement (ii) is similar.
	Suppose that $o_j^{-,r} + a_{rj} \ell_j \le b_r$ holds for all $r \in \M$ with $a_{rj} <0$. 
    Let  $x^*\in \X$ be a feasible solution with $x^*_j > \ell_j$ and  $\bar{x}=(x^*_1, \ldots,\ell_j, \ldots, x^*_n) $.
    For $r \in \M$ with 
    $a_{rj} \geq 0$, it follows that ${\sum_{s\in \N}} a_{rs} \bar{x}_s = \sum_{{s \in \N \backslash \{j\}}} a_{rs} x^*_s + a_{rj} \ell_j \le {\sum_{s\in \N}} a_{rs} x^*_s \le b_r$, where the last inequality follows from ${x^*} \in \X$; 
    for $r \in \M$ with $a_{rj} < 0$, it follows that 
    $$\sum_{s \in \N} a_{rs} \bar{x}_s = \sum_{{s \in \N \backslash \{j\}}} a_{rs} x^*_s + a_{rj} \ell_j \stackrel{(a)}{\leq} \max_{x_j > \ell_j} \left\{\sum\limits_{{s \in \N \backslash \{j\}}} a_{rs} x_s ~\mid~ x \in \X \right\} + a_{rj}\ell_j =o_{j}^{-,r}+a_{rj}\ell_j \le b_r,$$
    where (a) follows from $x^* \in \X$ and $x_j^* > \ell_j$.
    Therefore, $\bar{x} \in \X$ and $x_j$ is lower bound reachable.
    
    Now, suppose that $x_j$ is lower bound reachable. 
    If, otherwise, $o_{j}^{-,r_0}+a_{r_0j}\ell_j> b_{r_0}$  holds for some $r_0 \in \M$ with {$a_{r_0j}< 0$}, then by the definition of $o_{j}^{-,r_0}$ in \eqref{opt1}, there exists a point $x^*\in \X$ such that $o_{j}^{-,r_0}=\sum_{{s \in \N \backslash \{j\}}} a_{r_0s} x^*_s$, and thus $\sum_{{s \in \N \backslash \{j\}}} a_{r_0s} x^*_s+ a_{r_0j}  \ell_j > b_{r_0}$.
    This implies $\bar{x}:=(x_1^*, \ldots, \ell_j,\dots, x_n^*) \notin \X$, a contradiction with the fact that $x_j$ is lower bound reachable.
\end{proof}

{\cref{fixing-theorem} shows that determining whether $\{x_j\}$ are lower or upper bound
reachable can be done by solving the optimization problems \eqref{opt1} and \eqref{opt2}.
The following example further illustrates this.}

{\begin{example}[\cref{ConditionExample} continued]\label{ex2}
    Note that  constraint $0 \le 8x_1 + 4x_2 - x_3 \le 8$ in problem \eqref{condition-mip}  can be transformed into two constraints in the ``$\le$'' form:
    \begin{equation}\label{tmeq}
    \begin{aligned}
            & -8x_1 - 4x_2 + x_3 \le 0, \\
        & 8x_1 + 4x_2 - x_3 \le 8.
    \end{aligned}
    \end{equation}
    Let us use \cref{fixing-theorem} to show that $x_3$ is lower bound reachable.
    As only the second constraint in \eqref{tmeq} involves variable $x_3$ with a negative coefficient, we only need to consider the optimization problem 
    \begin{equation}
        \label{compute-o}
        o_{3}^{-,2} = \max_{x_3 > 0} \left\{ 8x_1 + 4x_2 ~\mid~ 0 \le 8x_1 + 4x_2 - x_3 \le 8,~x_1, x_2, x_3 \in \left\{0, 1\right\} \right\}.
    \end{equation}
    By simple computation, we obtain $o_{3}^{-,2} = 8$. 
    From $o_{3}^{-,2} + (-1)\ell_3 = 8 \leq 8 $, $x_3$ must be lower bound reachable.
\end{example}}

Unfortunately, although the above approach for checking whether $\{x_j\}$ are lower or upper bound reachable is computationally tractable,  solving the \MIP problems \eqref{opt1} and \eqref{opt2} could still be computationally demanding. 
In the next subsection, we will use the probing framework to find {upper bounds} for $o_j^{-,r}$ and $o_j^{+,r}$, thereby enabling a computationally efficient algorithm for detecting dual fixings.

\begin{remark}\label{remark1}
	{It is worth noting that  if  $x_j > \ell_j$ and  $x_j < u_j$ are removed from \eqref{opt1} and \eqref{opt2}, \cref{fixing-theorem} still holds.
	Indeed, let 
	\begin{align}
		& \bar{o}_j^{-,r} := \max \left\{\sum\limits_{s \in \N \backslash \{j\}} a_{rs} x_s ~\mid~ x \in \X \right\}, ~\forall~j \in \N,~r \in \M ~\text{with}~a_{rj} <0,\label{opt1a}\\
		& \bar{o}_j^{+,r} := \max \left\{\sum\limits_{s \in \N \backslash \{j\}} a_{rs} x_s ~\mid~ x \in \X \right\},~\forall~j \in \N,~r \in \M ~\text{with}~a_{rj} >0,\label{opt2a}
	\end{align}
	be obtained by removing $x_j > \ell_j$ and  $x_j < u_j$  from \eqref{opt1} and \eqref{opt2}, respectively.
	For any $j \in \N$ and $r \in \M$ with $a_{rj}<0$, as $x\in \X$, it must follow
	$\max_{x_j = \ell_j} \left\{\sum\limits_{s \in \N \backslash \{j\}} a_{rs} x_s ~\mid~ x \in \X \right\} + a_{rj} \ell_j \le b_r$, and thus $\bar{o}_j^{-,r} + a_{rj} \ell_j  \le b_r$ holds if and only if $o_j^{-,r} + a_{rj} \ell_j  \le b_r$ holds.
	Similarly, for any $j \in \N$ and $r \in \M$ with $a_{rj}>0$, $\bar{o}_j^{+,r}+ a_{rj} u_j  \le b_r$ holds if and only if $o_j^{+,r}+ a_{rj} u_j  \le b_r$ holds.
	As a result, \cref{fixing-theorem} still holds when problems \eqref{opt1} and \eqref{opt2} are replaced by the standard \MIP problems \eqref{opt1a} and \eqref{opt2a} (without strict inequalities).}
	 
	{However, in approximation approaches such as the one described in the next subsection, 
	the strict inequalities $x_j>\ell_j$ and $x_j < u_j$ 
	can tighten the feasible regions of problems \eqref{opt1} and \eqref{opt2}, which
	usually results in tighter upper bounds for $o_j^{-,r}$ and $o_j^{+,r}$ (than those for $\bar{o}_j^{-,r}$ and $\bar{o}_j^{+,r}$), thereby triggering more generalized dual fixings.}
\end{remark}


\subsection{Deriving {upper bounds} for $o_j^{-,r}$ and $o_j^{+,r}$ by probing}

To fix a variable $x_j$ by \cref{fixing-theorem} efficiently, here we use the probing framework to compute {upper bounds} for $o_j^{-,r}$ and $o_j^{+,r}$.
Specifically, letting $\ell^0$ and $u^0$ (respectively, $\ell^1$ and $u^1$) be the lower and upper bounds of variables $x$, obtained by 
probing on $x_k = 0$ (respectively, $x_k = 1$); 
see \cref{alg1}.
In other words, $\ell^0 \leq x \leq u^0$ holds for all $x \in \X$ with $x_k=0$, and $\ell^1 \leq x \leq u^1$ holds for all $x \in \X$ with $x_k=1$.
Then it follows that
\begin{equation}\label{estimate}
	\small
	\begin{aligned}
    o_j^{-,r}  = & \max_{x_j > \ell_j} \left\{\sum\limits_{{s \in \N \backslash \{j\}}} a_{rs} x_s ~\mid~ x \in \X \right\} \\
    \leq & \max\left\{\sum\limits_{{s \in \N \backslash \{j\}}, \,a_{rs} > 0} a_{rs} u^0_s + \sum\limits_{{s \in \N \backslash \{j\}},  \,a_{rs} < 0} a_{rs} \ell_s^0 ,
		\sum\limits_{{s \in \N \backslash \{j\}},  \,a_{rs} > 0} a_{rs} u^1_s + \sum\limits_{{s \in \N \backslash \{j\}},  \, a_{rs} < 0} a_{rs} \ell_s^1\right\}.
		\end{aligned}
\end{equation}
Note that {if $k = j$ (i.e., the probing variable is $x_j$)}, then {the strict inequality $x_j > \ell_j$ in problem \eqref{opt1} enforces a tighter {upper bound} for $o_j^{-,r}$}:
\begin{equation}\label{estimate2}
	\begin{aligned}
		o_j^{-,r}  = & \max_{x_j > \ell_j} \left\{\sum\limits_{{s \in \N \backslash \{j\}}} a_{rs} x_s ~\mid~ x \in \X \right\}
		\leq\sum\limits_{{s \in \N \backslash \{j\}}, \, a_{rs} > 0} a_{rs} u^1_s + \sum\limits_{{s \in \N \backslash \{j\}},  \,a_{rs} < 0} a_{rs} \ell_s^1.
	\end{aligned}
\end{equation}

{Given a fixed probing variable $x_k$,} a direct implementation of computing the {upper bounds} of $o_j^{-,r}$ (i.e., the right-hand side of \eqref{estimate} or \eqref{estimate2}) for all {$j \in \N$ and $r \in \M$ with $a_{rj}<0$} takes $\CO(|\N| \times \texttt{nnz})$ operations.
{However, this computational complexity can be reduced to $\CO(\texttt {nnz})$, as detailed in the following.
First, we compute}
\begin{equation}
	\label{compute-alpha}
	\alpha_r^0 = \sum\limits_{s \in \N, \, a_{rs} > 0} a_{rs} u^0_s + \sum\limits_{s \in \N,  \,a_{rs} < 0} a_{rs} \ell_s^0~\text{and}~\alpha_r^1 = \sum\limits_{s \in \N,  \,a_{rs} > 0} a_{rs} u^1_s + \sum\limits_{s \in \N,  \,a_{rs} < 0} a_{rs} \ell_s^1
\end{equation}
{for all $r \in \M$ with the complexity of $\CO(\texttt{nnz})$. 
Then, using} 
\begin{align}
	\label{estimateFast}
	\sum\limits_{{s \in \N \backslash \{j\}},  \,a_{rs} > 0} a_{rs} u^0_s + \sum\limits_{{s \in \N \backslash \{j\}}, \, a_{rs} < 0} a_{rs} \ell_s^0 =\alpha_r^0 - a_{rj} \ell_j^0, \\
	\sum\limits_{{s \in \N \backslash \{j\}},  \,a_{rs} > 0} a_{rs} u^1_s + \sum\limits_{{s \in \N \backslash \{j\}}, \, a_{rs} < 0} a_{rs} \ell_s^1=\alpha_r^1 - a_{rj} \ell_j^1,
\end{align}
{the right-hand side of \eqref{estimate} or \eqref{estimate2} reduces to $\max\{\alpha_r^0- a_{rj} \ell_j^0,~ \alpha_r^1- a_{rj} \ell_j^1\}$ or $\alpha_r^1- a_{rj} \ell_j^1$, respectively, and can be computed in constant time.
As a result, the upper bounds of $o_j^{-,r}$ for all {$j \in \N$ and $r \in \M$ with $a_{rj}<0$}  can be computed in $\CO(\texttt {nnz})$ time.}

{Similar {upper bounds} can be derived for $o_j^{+,r}$. Note that when probing on multiple binary variables $\{x_k\}_{k \in \mathcal{S}}$, we may obtain $|\mathcal{S}|$ upper bounds for $o_{j}^{-,r}$ and $o_{j}^{+,r}$ with the complexity of $\CO(|\mathcal{S}| \times \texttt{nnz})$. These bounds may be different, and thus}, we can choose the tightest ones, as to trigger more dual fixings.

{
\begin{example}[\cref{ConditionExample} continued]
	Let us use probing to derive an upper bound for $o_{3}^{-,2}$ in problem \eqref{compute-o}.
    Suppose that variable $x_1$ is probed on. 
    Fixing $x_1 = 1$ and applying domain propagation (on the constraints in problem \eqref{condition-mip}) yields $x_2 = 0$ and thus
   	$\ell^1 = \left(1,0,0\right)$ and $ u^1 = \left(1,0,1\right)$. 
   	Similarly, fixing $x_1=0$ and applying domain propagation, we obtain  $\ell^0 = \left(0,0,0\right)$ and $u^0 = \left(0,1,1\right)$. 
   	As a result, we can compute an upper bound for $o_{3}^{-,2}$:
    $$\max\left\{\alpha_2^0 - a_{23} \ell_3^0,~ \alpha_2^1 - a_{23} \ell_3^1\right\} = \max\{ 4- 0, 8 - 0\}=  8.$$
    From \cref{ex2}, this upper bound is tight and not greater than the right-hand side of the second constraint in \eqref{tmeq}. 
    Therefore, we can conclude that $x_3$ is lower bound reachable.
\end{example}}

\section{Numerical experiments}
\label{section-num}

In this section, we evaluate the performance impact of the two proposed methods in combining and enhancing probing and dual fixing in \cref{section-dfprobing,section-generalize}.
We implement the proposed methods in the open-source \MIP solver HiGHS \rev{$1.14.0$} \citep{Huangfu2018}.
Note that \rev{standard probing, dual fixing, and dual substitution \cite{Achterberg2019} (see also the last paragraph in \cref{section-dualfixing}) have all been implemented in HiGHS 1.14.0}, and will be included in all our computational experiments.
The experiments are conducted on a Linux cluster of \rev{Intel(R) Xeon(R) Platinum 8358 CPUs running at 2.60 GHz.}
We run \HiGHS using the serial version, with a time limit of 7200 seconds.


We test our algorithms on the MIPLIB 2017 benchmark testset \citep{Gleixner2021}. 
We use 5 random seeds for each of the $240$ problems, and each problem and seed combination is treated as an individual observation, referred to as an ``instance''. 
Note that the default setting of HiGHS reported incorrect objective values for \rev{one} instance, i.e., \texttt{map10} with seed 1, and after embedding dual fixing into the probing framework, HiGHS \rev{additionally} reported incorrect objective values for \rev{two instances: \texttt{map10} with seed 4 and \texttt{blp-ar98} with seed 1}\footnote{\rev{For these two instances, the objective values reported by HiGHS with dual fixing augmented probing are $-494.662795463$ and $6205.835712$, while the true objective values are $-495$ and $6205.2147104$.
In an attempt to rule out the possibility of introducing bugs when implementing the dual fixing augmented probing, we have outputted the presolved instances, and solved them by the default setting of CPLEX 20.1.0.0 with $5$ seeds. 
The results show that in all cases, CPLEX can return the \emph{correct} objective values: $-495$ and $6205.2147104$.}}.
In our experiments, we exclude these three instances, yielding a testbed of $1197$ instances used in Tables \ref{default-dual-1}--\ref{default-full-1}.



We first evaluate the performance {impact} of integrating dual fixing into the probing framework.
To do this, we compare the default setting of \HiGHS, denoted as \Default, with the setting \DFProbing where dual fixing is embedded into the probing framework;
see \cref{alg1}.
The results are summarized in \cref{default-dual-1}.
In \cref{default-dual-1}, the ``$\ge n$'' bracket contains instances that are solved by at least one setting under comparison and for which the slower
of the two used $\ge n$ seconds. 
{With the increasing of $n$, the brackets can exclude instances that are ``easy'' for both settings, forming a hierarchy of instance subsets with increasing difficulty \citep{Achterberg2013b}.}
Column \texttt{\#Ins} denotes the number of instances solved by at least one setting ({in total, \rev{763} instances can be solved by at least one setting, and the remaining \rev{434} instances cannot be solved to optimality by both settings}), and column \texttt{\#S} denotes the number of instances solved by the corresponding setting in each bracket. 
Columns \texttt{T} and \texttt{N} {under \Default or \DFProbing} are the shifted geometric {means} of {CPU} time {in seconds} and number of nodes, {with a shift of 1.
Under ``Compare'', we report the ratios of the shifted geometric means of the CPU times and number of nodes (a value less than $1.0$ denotes that \DFProbing achieves a better performance than \Default).}
We observe from \cref{default-dual-1} that embedding dual fixing into the probing framework enables \DFProbing to achieve a \rev{slightly} better performance.
Overall, \rev{the number of nodes returned by \DFProbing is {3.1\%} smaller than those returned by \Default, with a slight reduction on CPU time}.
However, for harder instances that are solved by at least $1000$ seconds, we can observe a CPU time reduction of \rev{$2.0\%$}, {showing that the performance is more pronounced on relatively difficult instances.}
For the \rev{$246$} affected instances where the solving path differs among two settings\footnote{{Here, we follow \cite{Achterberg2013b} to assume that the solving path is identical if both the number of nodes and the number of simplex iterations are identical for the two settings.}}, \rev{we can observe a CPU time reduction of $1.6\%$ and a node reduction of $9.7\%$, respectively.}

\begin{table}[t]
    \renewcommand{\arraystretch}{1.4}
	\addtolength{\tabcolsep}{-2pt}
    \centering
    \scriptsize
    \caption{Performance impact of embedding dual fixing into the probing framework.}
    \rev{\begin{tabular}{|l|r|r|r|r|r|r|r|r|r|r|r|r|} \hline
		&              &              \multicolumn{3}{c|}{\Default}  &           \multicolumn{3}{c|}{\DFProbing}  &\multicolumn{2}{c|}{Compare} &           \multicolumn{3}{c|}{Affected}      \\ \hline
Bracket   &      \texttt{\#Ins}    &       \texttt{\#S}    &         \texttt{T}    &         \texttt{N}    &       \texttt{\#S}    &         \texttt{T}    &         \texttt{N}    &         \texttt{T}    &         \texttt{N}    &    \texttt{\#Ins}    &         \texttt{T}    &          \texttt{N}     \\ \hline
$\ge$       0     &     763      &      751     &    240.64    &    817.81    &      760     &    239.92    &    792.58    &       0.997  &       0.969  &     246      &        0.984  &        0.903    \\ \hline
$\ge$      10     &     687      &      675     &    374.12    &   1437.22    &      684     &    372.61    &   1389.42    &       0.996  &       0.967  &     241      &        0.984  &        0.904    \\ \hline
$\ge$     100     &     506      &      494     &    862.88    &   4232.91    &      503     &    857.03    &   4081.48    &       0.993  &       0.964  &     200      &        0.980  &        0.908    \\ \hline
$\ge$    1000     &     251      &      239     &   2405.06    &  14188.79    &      248     &   2356.93    &  13752.49    &       0.980  &       0.969  &     118      &        0.953  &        0.931    \\ \hline
    \end{tabular}}
	\label{default-dual-1}
\end{table}


To better understand where the improvements come from, {we further compare  the total number of bound changes (where for a variable $x_k$, the number of bound changes is defined by $\sum_{v=0}^1 (|\{ j \in \N\backslash \{k\} \, : \,\ell_j^v\neq \ell_j \}|+|\{ j \in \N\backslash \{k\} \, : \,u_j^v\neq u_j \}|)$) derived by the first round of probing (i.e., the presolve round where probing is first called \citep{Gleixner2023}) with and without dual fixing embedded.
The total number of bound changes reflects the number of reductions in R1--R4 detected by probing, i.e., the larger the total number of bound changes, the more the reductions in R1--R4.
The results are summarized in 
\cref{default-dual-first-round}. 
Column \Impl reports the shifted geometric mean of the total number of bound changes (with a shift of $100$) detected in the first round of probing.
When counting the bound changes, we do not count the disregarded ones reasoning from dual fixing of variable $x_j$ with $c_j=0$; see \cref{section-embed}.
Column \texttt{\Tprobing}  reports the shifted geometric mean of the CPU time in seconds (with a shift of $1$) spent in the first round of probing.
Note that each bracket ``$\ge n$'' in \cref{default-dual-first-round} contains the same instances as in \cref{default-dual-1}.}
From \cref{default-dual-first-round}, we observe that the computational effort spent in implementing dual fixing into the probing framework is fairly small; overall, it only leads to an increase of \rev{$7.9\%$} probing time. 
In contrast, dual fixing augmented probing can detect \rev{$20.2\%$} more bound changes than those detected by the vanilla probing, which enables probing to derive more reductions in R1--R4, and thus further renders the branch-and-cut algorithm to explore a smaller search tree and achieve a better overall performance; see \cref{default-dual-1}.

\begin{table}[t]
    \renewcommand{\arraystretch}{1.4}
	\addtolength{\tabcolsep}{-2pt}
	\centering
	\scriptsize
	\caption{Comparison of {the bound changes detected in the first round of probing with and without dual fixing embedded.}}
    \rev{\begin{tabular}{|l|r|r|r|r|r|r|r|r|r|r|r|r|r|r|r|r|} \hline
		&              \multicolumn{2}{c|}{\Default}  &           \multicolumn{2}{c|}{\DFProbing}  &               \multicolumn{2}{c|}{Compare} &           \multicolumn{2}{c|}{Affected}      \\ \hline
Bracket        &   \quad \Impl     & \Tprobing   &    \Impl    & \Tprobing  &   \Impl  & \Tprobing  &   \Impl  & \Tprobing   \\ \hline
$\ge$        0     &   2344.72    &      0.44    &  2817.77    &      0.48    &      1.202  &       1.079  &             1.717  &        1.119   \\ \hline
$\ge$       10     &   2520.88    &      0.47    &  3089.82    &      0.51    &      1.226  &       1.082  &             1.736  &        1.120   \\ \hline
$\ge$      100     &   2297.20    &      0.47    &  2900.96    &      0.50    &      1.263  &       1.082  &             1.777  &        1.115   \\ \hline
$\ge$     1000     &   2001.84    &      0.50    &  2814.88    &      0.55    &      1.406  &       1.091  &             2.029  &        1.153   \\ \hline
    \end{tabular}}
	\label{default-dual-first-round}
\end{table}

\begin{table}[t]
    \renewcommand{\arraystretch}{1.4}
    \addtolength{\tabcolsep}{-2pt}
    \centering
    \scriptsize
      \caption{Performance impact of using probing to detect  (generalized) dual fixings {(with \Default as the baseline)}.}
    \label{default-onlypartial-1}
    \rev{\begin{tabular}{|l|r|r|r|r|r|r|r|r|r|r|r|r|} \hline
                      &          &\multicolumn{3}{c|}{\Default }              &     \multicolumn{3}{c|}{\GDF }       &  \multicolumn{2}{c|}{Compare}  &  \multicolumn{3}{c|}{Affected}   \\ \hline
        Bracket       &  \texttt{\#Ins}    &  \texttt{\#S}        & \texttt{T}              & \texttt{N}             &  \texttt{\#S}        & \texttt{T}              & \texttt{N}              &  \texttt{T}        & \texttt{N}         &  \texttt{\#Ins}  & \texttt{T}        & \texttt{N}      \\ \hline
        $\ge$       0     &     755      &      751     &    232.11    &    806.13    &      750     &    232.97    &    809.15    &         1.004  &       1.004   &      58      &       1.019  &        1.059    \\ \hline
        $\ge$      10     &     679      &      675     &    361.29    &   1416.78    &      674     &    362.55    &   1422.70    &         1.003  &       1.004   &      58      &       1.019  &        1.059    \\ \hline
        $\ge$     100     &     491      &      487     &    860.51    &   4584.30    &      486     &    863.47    &   4610.92    &         1.003  &       1.006   &      58      &       1.019  &        1.059    \\ \hline
        $\ge$    1000     &     236      &      232     &   2421.77    &  16024.50    &      231     &   2418.25    &  15940.83    &         0.999  &       0.995   &      36      &       0.987  &        0.957    \\ \hline
        \end{tabular}}
\end{table}

Next, we evaluate the performance impact of using the probing framework to detect (generalized) dual fixings; see \cref{section-generalize}.  
The comparison results of the default setting of \HiGHS and the one with the generalized dual fixing (denoted as \GDF) are summarized in \cref{default-onlypartial-1}. 
\rev{Overall, using \Default as a baseline, the impact of the proposed generalized dual fixing  is neutral, affecting only 58 instances.}
We also performed an additional experiment by applying the generalized dual fixing on top of \DFProbing (denoted as \All), and compared it with \DFProbing (where dual fixing augmented probing was implemented).
\rev{From the results in \cref{dual-all}, using \DFProbing as a baseline, the impact of the proposed generalized dual fixing  is still neutral.  
However, for the 43 affected instances, a CPU time reduction of 2.7\% and a node reduction of 1.5\% can be observed.}


\begin{table}[t]
	\renewcommand{\arraystretch}{1.4}
	\addtolength{\tabcolsep}{-2pt}
	\centering
	\scriptsize
	\caption{{Performance impact of using probing to detect  (generalized) dual fixings (with \DFProbing as the baseline)}.}
	\rev{\begin{tabular}{|l|r|r|r|r|r|r|r|r|r|r|r|r|} \hline
			&         & \multicolumn{3}{c|}{\DFProbing }    &                    \multicolumn{3}{c|}{\All }    &                          \multicolumn{2}{c|}{Compare}  &  \multicolumn{3}{c|}{Affected}   \\ \hline
			Bracket          &  \texttt{\#Ins}  &  \texttt{\#S}         & \texttt{T}            & \texttt{N}            &  \texttt{\#S}       & \texttt{T}           & \texttt{N}            & \texttt{T}             & \texttt{N}             &  \texttt{\#Ins}     & \texttt{T}       & \texttt{N}        \\ \hline
			$\ge$       0     &     760      &      760     &    236.71    &    800.00    &      757     &    236.35    &    799.36    &       0.998  &       0.999  &      43      &       0.973  &        0.985   \\ \hline
            $\ge$      10     &     684      &      684     &    367.79    &   1393.51    &      681     &    367.34    &   1392.26    &       0.999  &       0.999  &      41      &       0.971  &        0.984   \\ \hline
            $\ge$     100     &     498      &      498     &    871.14    &   4340.17    &      495     &    869.28    &   4334.83    &       0.998  &       0.999  &      41      &       0.971  &        0.984   \\ \hline
            $\ge$    1000     &     244      &      244     &   2436.93    &  14413.51    &      241     &   2421.81    &  14306.47    &       0.994  &       0.993  &      27      &       0.964  &        0.928   \\ \hline
	\end{tabular}}
	\label{dual-all}
\end{table}

\begin{table}[t]
	\renewcommand{\arraystretch}{1.4}
	\addtolength{\tabcolsep}{-2pt}
	\centering
	\scriptsize
	\caption{Performance impact of embedding dual fixing into the probing framework and using probing to detect generalized dual fixings.}
	\label{default-full-1}
	\rev{\begin{tabular}{|l|r|r|r|r|r|r|r|r|r|r|r|r|} \hline
		            &          &    \multicolumn{3}{c|}{\Default }              &  \multicolumn{3}{c|}{\All }                 &   \multicolumn{2}{c|}{Compare} &  \multicolumn{3}{c|}{Affected}      \\ \hline
		Bracket     &  \texttt{\#Ins}  &  \texttt{\#S}     & \texttt{T}              & \texttt{N}               &  \texttt{\#S}           & \texttt{T}            & \texttt{N}           & \texttt{T}        &    \texttt{N}     &  \texttt{\#Ins}  &   \texttt{T}       & \texttt{N}       \\ \hline
		$\ge$       0     &     763      &      751     &    240.64    &    796.51    &      757     &    239.55    &    770.64    &         0.995  &       0.968  &     271      &       0.983  &        0.907    \\ \hline
        $\ge$      10     &     687      &      675     &    374.12    &   1399.16    &      681     &    372.15    &   1350.07    &         0.995  &       0.965  &     266      &       0.984  &        0.908    \\ \hline
        $\ge$     100     &     506      &      494     &    862.88    &   4107.22    &      500     &    855.18    &   3950.00    &         0.991  &       0.962  &     225      &       0.978  &        0.912    \\ \hline
        $\ge$    1000     &     253      &      241     &   2380.09    &  14172.72    &      247     &   2323.93    &  13644.77    &         0.976  &       0.963  &     136      &       0.957  &        0.927    \\ \hline
	\end{tabular}}
\end{table}

Now, we present the computational results in \cref{default-full-1} to demonstrate the effectiveness of the two proposed techniques in combining and enhancing probing and dual fixing; see Sections 2 and 3.
As observed in \cref{default-full-1}, the \HiGHS with the two proposed techniques can achieve a better performance than \Default.
\rev{For instances that are solved by at least $1000$ seconds, we can observe a CPU time reduction of $2.4\%$ and  a node reduction of $3.7\%$;
for the 271 affected instances, we can observe a CPU time reduction of 1.7\% and a node reduction of 9.3\%.}
Note that compared with applying the individual technique, applying both the two techniques in combining probing and dual fixing enables HiGHS to find reductions on more problem instances \rev{(271)}, and makes a better overall performance.

\section{Conclusions {and future work}}
\label{section-conclusion}
In this paper, we have investigated the probing and dual fixing techniques, and developed two methods to combine and enhance these two powerful techniques for MIP problems. 
The first method embeds dual fixing into the probing framework to detect more variable implications for probing.
The second method generalizes the classic dual fixing technique where more variable fixings can be detected, and uses the probing framework for a quick verification.
Through computational experiments on the MIPLIB 2017 benchmark instances, we demonstrate \rev{the potential of the two proposed techniques on \HiGHS}.


Regarding the proposed generalized dual fixing, we have used the probing framework to compute upper bounds  for $o_j^{-,r}$ and $o_j^{+,r}$  in \eqref{opt1} and \eqref{opt2} and  detect variable fixings by verifying the conditions in \cref{fixing-theorem}. 
It would be interesting to explore other alternatives to derive upper bounds for $o_j^{-,r}$ and $o_j^{+,r}$ towards a more effective implementation of the generalized dual fixing.

\begin{acknowledgement}
    The authors would like to thank the editors and two anonymous reviewers for their insightful comments that significantly improved the quality of our work.
\end{acknowledgement}




\bibliographystyle{spmpsci}      
\bibliography{shorttitles, library-svjour3}

\end{document}